\documentclass{amsart}
\usepackage{color}
\usepackage{amsmath,amsfonts,amsthm,amssymb,amscd}
\usepackage[all]{xy}
\newtheorem{thm}{Theorem}[section]
\newtheorem{crl}[thm]{Corollary}
\newtheorem{lmm}[thm]{Lemma}
\newtheorem{prp}[thm]{Proposition}
\theoremstyle{definition}
\newtheorem{dfn}[thm]{Definition}
\newtheorem{ex}[thm]{Example}
\theoremstyle{remark}
\newtheorem*{rmk}{Remark}

\numberwithin{equation}{section}
\allowdisplaybreaks
\DeclareMathOperator{\Diff}{\bf Diff}
\DeclareMathOperator{\Top}{\bf Top}

\DeclareMathOperator{\Ho}{\bf Ho}

\DeclareMathOperator{\Cinfty}{{\mathit C}^{\infty}}
\DeclareMathOperator{\ev}{ev}

\DeclareMathOperator{\Loops}{Loops}

\newcommand{\R}{\mathbf{R}}
\newcommand{\I}{\mathcal{I}}
\newcommand{\J}{\mathcal{J}}
\newcommand{\K}{\mathcal{K}}
\newcommand{\semicolon}{\,;}
\newcommand{\bI}{\partial I}

\newcommand{\abs}[1]{\lvert{#1}\rvert}
\newcommand{\rel}[1]{\ \mbox{\rm rel}\ #1}
\begin{document}
\title[The Quillen model structure on the category $\Diff$] 
      {The Quillen model structure on the category of diffeological spaces} 
\author[T.~Haraguchi] 
       {Tadayuki HARAGUCHI}
\address[Tadayuki Haraguchi]
        {Faculty of Education for Human Growth\\
         Nara Gakuen University\\
         Nara, 631-0003 Japan}
\email{t-haraguchi@naragakuen-u.jp}

\author[K.~Shimakawa] 
       {Kazuhisa SHIMAKAWA}
\address[Kazuhisa Shimakawa]
        {
         Okayama University\\
         Okayama, 700-8530 Japan}
\email{kazu@math.okayama-u.ac.jp}


\thanks{This research was partially supported by JSPS KAKENHI Grant Number JP18K03279.}

\begin{abstract}
We construct on the category of diffeological spaces a Quillen model structure having smooth weak homotopy equivalences as the class of weak equivalences.
\end{abstract}
\maketitle
\section{Introduction}
The theory of model category was introduced by Quillen
(cf.~\cite{QuiH} and \cite{QuiR}), and is a crucial notion that
forms the framework of modern homotopy theory.  A model category
is just a category with three specified classes of morphism,
called fibrations, cofibrations and weak equivalences, which
satisfy several axioms that are deliberately reminiscent of
typical properties appearing in homotopy theory of topological
spaces.  It is shown in \cite[8.3]{Spa} that the category
$\bf Top$ of topological spaces has, so called, the Quillen model
structure, under which a map $f \colon X \to Y$ is defined to be

\begin{enumerate}
\item a weak equivalence if $f$ is a weak homotopy equivalence
  \cite[p.404]{Spanier},
\item a fibration if $f$ is a Serre fibration {\cite{Serre}}, and
\item a cofibration if $f$ has the left lifting property with
  respect to trivial fibrations.
\end{enumerate}

The objective of this paper is to introduce a model category
structure on the category $\Diff$ of diffeological
spaces 
which closely resembles the original Quillen model structure of
$\Top$ as in the following sense:
A smooth map between diffeological spaces is a weak equivalence
if it induces isomorphisms between smooth homotopy groups, and is
a fibration if it satisfies (a sort of) homotopy lifting
properties for pairs of cubical complexes $(I^n,J^{n-1})$, where
$I^n$ is the unit $n$-cube and
$J^{n-1} = \bI^{n-1} \times I \cup I^{n-1} \times \{1\}$.
Our construction is directly connected to the smooth homotopy
theory of diffeological spaces, and seems to be suited to combine
homotopy theoretical methods with methods of differential
topology and geometry.
In a subsequent part of this work \cite{HS1} (see also our
original preprint \cite{HS}), we show that our model structure is
Quillen equivalent to the Quillen model structure of $\Top$ under
the adjunction $\Top \rightleftarrows \Diff$.

The paper is organized as follows.  In Section 2 we study
homotopy sets of smooth maps between diffeological spaces.  In
particular, smooth homotopy groups $\pi_n(X,x_0)$ and relative
homotopy groups $\pi_n(X,A,x_0)$ are defined to be the homotopy
sets of smooth maps $(I^n,\bI^n) \to (X,x_0)$ and
$(I^n,\bI^n,J^{n-1}) \to (X,A,x_0)$, respectively.
Our definition of homotopy groups is different from, but turns
out to be equivalent to, the one given by \cite{Zem}.

Since $J^{n-1}$ is not a smooth retract of $I^n$, the treatment
of smooth homotopy groups is slightly harder compared with the
case of continuous homotopy groups.  Still, as we shall see in
Section~3, basic properties of continuous homotopy groups mostly
hold in the smooth case.  Especially, there exists a
\emph{homotopy long exact sequence} associated with a pair of
diffeological spaces (cf.\ Proposition~\ref{prp:homotopy exact
  sequence}).



In section 4 we introduce the notion of fibrations and
investigate its basic properties.
Briefly, fibration is a smooth version of Serre fibration and is
characterized by the right lifting property (in certain
restricted sense) with respect to the inclusions
$J^{n-1} \to I^n$, and its examples include trivial maps,
diffeological fiber bundles, and (tame version of) path-loop
fibrations (see Example~\ref{examples of fibrations}).  In
particular, every diffeological space $X$ is fibrant because the
constant map $X \to *$ is a fibration (cf.\
Corollary~\ref{crl:admissible maps are extendable}).

Based on the results obtained in preceding sections, we prove in
Section~5 that the category $\Diff$ has a model category
structure such that a smooth map is (i) a weak equivalence if it
is a weak homotopy equivalence, (ii) a fibration if it is a
fibration defined in the preceding section, and (iii) a
cofibration if it has the left lifting property with respect to
trivial fibrations.
The main theorem (Theorem~\ref{thm:model structure}) is proved by
straightforward verification of axioms given by Dwyer-Spalinski
\cite{Spa} except for \textbf{MC5}, that is, a factorization of a
smooth map $f$ in two ways: (i) $f = p \circ i$, where $i$ is a
cofibration and $p$ is a trivial fibration, and (ii)
$f = p \circ i$, where $i$ is a trivial cofibration and $p$ is a
fibration.
The lack of smooth retraction of $I^n$ onto $J^{n-1}$ makes,
again, the verification of \textbf{MC5} far more complicated than
the case of $\Top$.

In Section 6 we complete the proof of our main theorem
by giving a detailed verification of \textbf{MC5}.

Several authors have attempted the construction of a model
category related to diffeological spaces.  In \cite{DE},
{Dan} Christensen and {Enxin} Wu discuss smooth homotopy theory with
respect to behaviors of fibrations, cofibrations and weak
equivalences.  However, they have yet to introduce a model
structure.  In \cite{Wu}, Enxin Wu presents a cofibrantly
generated model structure on the category of diffeological chain
complexes.  In \cite{Kihara}, Hiroshi Kihara proves that there
exists a cofibrantly generated model structure (cf.~\cite[Theorem
1.3 and Lemma 9.6]{Kihara}) in which fibrations are smooth maps
enjoying the right lifting property with respect to inclusions
$\Lambda^{p}_{k} \to \Delta^{p}$, where $\Delta^{p}$ and
$\Lambda^{p}_{k}$ are the $p$-simplex and its horn, respectively,
equipped with certain (non-standard) diffeology.
{ Thus there are two model structures which
  give rise to distinct model categories $\Diff_{\text{HS}}$
  and $\Diff_{\text{Kihara}}$.
  We constructed in \cite{HS} a Quillen equivalence
  $T \!:\! \Diff_{\text{HS}} \rightleftarrows \Top \!:\!
  D$.
  %
  On the other hand, Kihara constructed in \cite[Theorem
  1.5]{Kihara4} a Quillen equivalence
  $\abs{\ }_D \!:\! \mathcal{S} \rightleftarrows
  \Diff_{\text{Kihara}} \!:\! S^D$, where $\mathcal{S}$ is
  the model category of simplicial sets.
  Thus, by passing to homotopy categories, there is a chain
  of categorical equivalences
  \[
    \Ho(\Diff_{\text{HS}}) \rightarrow \Ho(\Top) \leftarrow
    \Ho(\mathcal{S}) \rightarrow \Ho(\Diff_{\text{Kihara}})
  \]
  in which the middle arrow is induced by the Quillen
  equivalence
  $\abs{\ } \!:\! \mathcal{S} \rightleftarrows \Top \!:\!
  S$.  }
%

The authors wish to thank Dan Christensen and Hiroshi Kihara for
helpful discussions while preparing the article.  In particular,
Kihara pointed out in \cite{Kihara} that there exists a gap in
the proof of Lemma 5.6 of our original manuscript \cite{HS}.  We
correct this gap by suitably modifying the infinite gluing
construction (see Section 6).

\section{Diffeological spaces and Homotopy sets}
%
In this section we introduce homotopy sets of a diffeological
space in a slightly different manner than the one given in
\cite[Chapter 5]{Zem} and \cite[Section 3]{DE}.

{ Recall from \cite{Souriau} that a
  diffeological space consists of a set and its
  \emph{diffeology}, that is, the family of \emph{plots}
  (defined on Euclidean open sets) satisfying the conditions
  similar to, but much more relaxed than, the charts of a
  smooth manifold.  Any subset of a Euclidean space can be
  regarded as a diffeological space with respect to the
  \emph{standard diffeology} consisting of all smooth
  parametrizations.
  If $X$ and $Y$ are diffeological spaces, then a smooth map
  from $X$ to $Y$ is a set map $f \colon X \to Y$ such that
  for any plot $P \colon U \to X$ of $X$ the composition
  $f \circ P \colon U \to Y$ is a plot of $Y$.
  Denote by $\Cinfty(X,Y)$ the set of smooth maps from $X$
  to $Y$.  Then $\Cinfty(X,Y)$ can be regarded as a
  diffeological space with respect to the coarsest (i.e.\
  weakest) diffeology such that the evaluation map
  \[
    \ev \colon C^{\infty}(X,Y) \times X \to Y, \quad
    \ev(f,x)=f(x)
  \]
  is smooth.  
  Let $\Diff$ be the category with diffeological spaces as
  objects and smooth maps as morphisms.  The theorem below
  plays an essential role in our ongoing argument.  See
  \cite{Zem} for the proof.

  \begin{thm}\label{thm:properties_of_Diff}
    The following hold:
    \begin{enumerate}
    \item $\Diff$ is self-enriched in the sense that the
      inclusion $X \to \Cinfty(X,X)$ and the composition
      $\Cinfty(Y,Z) \times \Cinfty(X,Y) \to \Cinfty(X,Z)$
      are smooth for all $X,\, Y,\, Z \in \Diff$.
    \item $\Diff$ is closed under small limits and colimits.
    \item $\Diff$ is a cartesian closed category with
      $\Cinfty(X,Y)$ as exponential objects; in fact, there
      is a natural isomorphism
      \[
        \alpha \colon \Cinfty(X \times Y,Z) \to
        \Cinfty(X,\Cinfty(Y,Z))
      \]
      given by the formula: $\alpha(f)(x)(y) = f(x,y)$ for
      $x \in X$ and $y \in Y$.
    \end{enumerate}
  \end{thm}
}%
%

Now, let $\R$ be the real line equipped with the standard
diffeology, and let $I$ be the unit interval $[0,1] \subset \R$
equipped with the subspace diffeology.  Suppose
$f_0,\, f_1 \colon X \to Y$ are smooth maps between diffeological
spaces.  We say that $f_0$ and $f_1$ are homotopic, written
$f_0 \simeq f_1$, if there is a smooth map
$F \colon X \times I \to Y$ such that $F(x,0) = f_0(x)$ and
$F(x,1) = f_1(x)$ hold for every $x \in X$.  Such a smooth map
$F$ is called a homotopy between $f_0$ and $f_1$.
A map $f \colon X \to Y$ is called a homotopy equivalence if
there is a smooth map $g \colon Y \to X$ satisfying
\[
  g \circ f \simeq 1 \colon X \to X, \quad f \circ g \simeq 1
  \colon Y \to Y.
\]
We say that $X$ and $Y$ are homotopy equivalent, written
$X \simeq Y$, if there exists a homotopy equivalence
$f \colon X \to Y$.  Let $X$ be a diffeological space and $A$ a
subspace of $X$.  Then $A$ is called a \emph{retract} of $X$ if
there exists a smooth map $\gamma \colon X \to A$, called a
\emph{retraction}, which restricts to the identity on $A$.
%
If, moreover, there exists a homotopy $H \colon X \times I \to X$
such that
\[
  H(x,0)=x,\ H(x,1)=\gamma(x),\ H(a,t)=a \quad (x \in X,\ a \in
  A)
\]
then $A$ is called a \emph{deformation retract} of $X$ and
$\gamma$ a \emph{deformation retraction}.

We show that the notion of homotopy introduced above is
equivalent to the one given in \cite[Chapter 5]{Zem} and
\cite[Section 3]{DE}.
{%
  Recall from \cite[{$\S 4$}]{JM} that there is a
  non-decreasing smooth function $\lambda \colon \R \to I$
  which satisfies $\lambda(t) = 0$ for $t \leq 0$,
  $\lambda(t) = 1$ for $1\leq t$,
  $\lambda(1-t) = 1-\lambda(t)$ for every $t$, and is
  strongly increasing on $[0,1]$.  }%

\begin{prp}
  Let $f_0,\, f_1 \colon X \to Y$ be smooth maps.  Then $f_0$ and
  $f_1$ are homotopic if and only if there exists a smooth map
  $G \colon X \times \R \to Y$ such that $G(x,0) = f_0(x)$ and
  $G(x,1) = f_1(x)$ hold for every $x \in X$.
\end{prp}

\begin{proof}
  Suppose there is a smooth map $G \colon X \times \R \to Y$ such that
  $G(x,0) = f_0(x)$ and $G(x,1) = f_1(x)$ hold for every $x \in X$.
  Then the restriction of $G$ to $X \times I$ gives a homotopy
  $f_0 \simeq f_1$.
  On the other hand, if there is a homotopy
  $F \colon X \times I \to Y$ between $f_0$ and $f_1$ then the
  composition $G = F\circ(1 \times \lambda)$ is a smooth map
  $X \times \R \to Y$ satisfying $G(x,0) = f_0(x)$ and
  $G(x,1) = f_1(x)$.
\end{proof}

Suppose $F$ is a homotopy between $f_0,\, f_1 \colon X \to Y$ and $G$
a homotopy between $f_1,\, f_2 \colon X \to Y$.  Let us define
$F * G \colon X \times I \to Y$ by the formula
\[
  F * G(x,t) =
  \begin{cases}
    F(x,\lambda(3t)), & 0 \leq t \leq 1/2,
    \\
    G(x,\lambda(3t-2)), & 1/2 \leq t \leq 1.
\end{cases}
\]
Then $F * G$ is smooth all over $X \times I$, hence gives a homotopy
between $f_0$ and $f_2$.  It follows that the relation ``$\simeq$'' is
an equivalence relation.  The resulting equivalence classes are called
homotopy classes.

In particular, if $P$ consists of a single point then smooth maps from
$P$ to $X$ are just the points of $X$ and their homotopies are smooth
paths $I \to X$.

\begin{dfn}
  Given a diffeological space $X$, we denote by $\pi_0{X}$ the set of
  path components of $X$, that is, equivalence classes of points of
  $X$, where $x$ and $y$ are equivalent if there is a smooth path
  $\alpha \colon I \to X$ such that $\alpha(0) = x$ and
  $\alpha(1) = y$ hold.
\end{dfn}
For given pairs of diffeological spaces $(X,X_1)$ and
$(Y,Y_1)$, put
\[
  [X,X_1 \semicolon Y,Y_1] = \pi_0\Cinfty((X,X_1),(Y,Y_1)),
\]
where $\Cinfty((X,X_1),(Y,Y_1))$ is the subspace of $\Cinfty(X,Y)$
consisting of maps of pairs $(X,X_1) \to (Y,Y_1)$.
Similarly, we put
\[
  [X,X_1,X_2 \semicolon Y,Y_1,Y_2] =
  \pi_0\Cinfty((X,X_1,X_2),(Y,Y_1,Y_2)),
\]
where $\Cinfty((X,X_1,X_2),(Y,Y_1,Y_2))$ is the subspace
consisting of maps of triples.  Clearly, we have
$[X,X_1 \semicolon Y,Y_1] = [X,X_1,\emptyset \semicolon
Y,Y_1,\emptyset]$.
{ Then, by using the cartesian closedness of
  $\Diff$ 
  we can prove the following.  }

\begin{prp}
  The elements of $[X,X_1,X_2 \semicolon Y,Y_1,Y_2]$ are in one-to-one
  correspondence with the homotopy classes of maps
  $(X,X_1,X_2) \to (Y,Y_1,Y_2)$.
\end{prp}

We also have the following proposition.

\begin{prp}
  \label{prp:mapping space is homotopy invariant}
  Suppose $f \colon (X,X_1,X_2) \to (Y,Y_1,Y_2)$ is a homotopy
  equivalence.  Then for every $(Z,Z_1,Z_2)$ the precomposition
  and postcomposition by $f$ induce homotopy equivalences
  \[
  \begin{split}
    f^* &\colon \Cinfty((Y,Y_1,Y_2),(Z,Z_1,Z_2)) \to
    \Cinfty((X,X_1,X_2),(Z,Z_1,Z_2))
    \\
    f_* &\colon \Cinfty((Z,Z_1,Z_2),(X,X_1,X_2)) \to
    \Cinfty((Z,Z_1,Z_2),(Y,Y_1,Y_2))
  \end{split}
  \]
\end{prp}

\begin{proof}
  For any $\boldsymbol{X} = (X,X_1,X_2)$ and fixed
  $\boldsymbol{Z} = (Z,Z_1,Z_2)$, put
  \[
    F\boldsymbol{X} = \Cinfty(\boldsymbol{X},\boldsymbol{Z}) =
    \Cinfty((X,X_1,X_2),(Z,Z_1,Z_2)).
  \]
  We shall show that the contravariant functor $F$ from the category
  of triples of diffeological spaces to $\Diff$ preserves homotopies.
  This of course implies that
  $f^* \colon F\boldsymbol{Y} \to F\boldsymbol{X}$ is a homotopy
  equivalence if so is $f \colon \boldsymbol{X} \to \boldsymbol{Y}$.

  The contravariant functor $F$ is enriched in the sense that the map
  \[
    \Cinfty(\boldsymbol{X},\boldsymbol{Y}) \to
    \Cinfty(F\boldsymbol{Y},F\boldsymbol{X}),
  \]
  which takes $f \colon \boldsymbol{X} \to \boldsymbol{Y}$ to the
  induced map $f^* \colon F\boldsymbol{Y} \to F\boldsymbol{X}$, is
  smooth.  This follows from Theorem~\ref{thm:properties_of_Diff} because the map above is adjoint to the composition
  $\Cinfty(\boldsymbol{Y},\boldsymbol{Z}) \times
  \Cinfty(\boldsymbol{X},\boldsymbol{Y}) \to
  \Cinfty(\boldsymbol{X},\boldsymbol{Z})$.

  Suppose $h \colon \boldsymbol{X} \times I \to \boldsymbol{Y}$ is a
  homotopy between $f$ and $g$.  Then by
  Theorem~\ref{thm:properties_of_Diff} together with the enrichedness
  of $F$ the composite map
  \[
    I \to \Cinfty(\boldsymbol{X},\boldsymbol{Y}) \to
    \Cinfty(F\boldsymbol{Y},F\boldsymbol{X}),
  \]
  which takes $t \in I$ to
  $h_t^* \colon F\boldsymbol{Y} \to F\boldsymbol{X}$, is smooth.
  Thus, by passing to the adjoint again, we get a smooth map
  $F\boldsymbol{Y} \times I \to F\boldsymbol{X}$ giving a homotopy
  between $f^*$ and $g^*$.

  Quite similarly, we can prove that the covariant functor
  $\boldsymbol{X} \to \Cinfty(\boldsymbol{Z},\boldsymbol{X})$
  preserves homotopies.
\end{proof}

\begin{crl}
  \label{crl:homotopy set is homotopy invariant}
  The homotopy set $[X,X_1,X_2 \semicolon Y,Y_1,Y_2]$ is homotopy
  invariant with respect to both $(X,X_1,X_2)$ and $(Y,Y_1,Y_2)$.
\end{crl}

We are now ready to define the $n$-th homotopy set of a
diffeological space.
%
%
Let $\bI^n$ be the boundary of the unit $n$-cube
$I^n \subset \R^n$, and let
\[
  J^{n-1} = \bI^{n-1} \times I \cup I^{n-1} \times \{1\} \quad (n
  \geq 1).
\]

\begin{dfn}
  Given a pointed diffeological space $(X,x_0)$, we put
  \[
    \pi_n(X,x_0) = [I^n,\bI^n \semicolon X,x_0],\quad n \geq 0.
  \]
  Similarly, given a pointed pair of diffeological spaces $(X,A,x_0)$,
  we put
  \[
    \pi_n(X,A,x_0) = [I^n,\bI^n,J^{n-1} \semicolon X,A,x_0],\quad n
    \geq 1.
  \]
\end{dfn}

For $n \geq 1$, $\pi_n(X,x_0)$ is isomorphic to $\pi_n(X,x_0,x_0)$,
and $\pi_0(X,x_0)$ is isomorphic to the set of path components
$\pi_0{X}$, regardless of the choice of basepoint $x_0$.  Note,
however, that we consider $\pi_0(X,x_0)$ as a pointed set with
basepoint $[x_0] \in \pi_0{X}$.
We now introduce a group structure on $\pi_n(X,A,x_0)$.  Suppose
$\phi$ and $\psi$ are smooth maps from $(I^n,\bI^n,J^{n-1})$ to
$(X,A,x_0)$.  If $n \geq 2$, or if $n \geq 1$ and $A = x_0$, then
there is a smooth map $\phi * \psi \colon I^n \to X$ which takes
$(t_1,t_2,\dots,t_n) \in I^n$ to
\[
  \begin{cases}
    \phi(\lambda(3t_1),t_2,\dots,t_n), & 0 \leq t_1 \leq 1/2
    \\
    \psi(\lambda(3t_1-2),t_2,\dots,t_n), & 1/2 \leq t_1 \leq 1.
  \end{cases}
\]
It is clear that $\phi * \psi$ defines a map of triples
$(I^n,\bI^n,J^{n-1}) \to (X,A,x_0)$, and there is a multiplication on
$\pi_n(X,A,x_0)$ given by the formula
\[
  [\phi] \cdot [\psi] = [\phi * \psi] \in \pi_n(X,A,x_0).
\]

\begin{prp}
  \label{prp:group structure}
  With respect to the multiplication
  $([\phi],[\psi]) \mapsto [\phi]\cdot[\psi]$, the homotopy set
  $\pi_n(X,A,x_0)$ is a group if $n \geq 2$ or if $n \geq 1$ and
  $A = x_0$, and is an abelian group if $n \geq 3$ or if $n \geq 2$
  and $A = x_0$.
  Moreover, for every
  smooth map $f \colon (X,A,x_0) \to (Y,B,y_0)$ the induced map
  \[
    f_* \colon \pi_n(X,A,x_0) \to \pi_n(Y,B,y_0)
  \]
  is a group homomorphism whenever its source and target are groups.
\end{prp}
{
\begin{rmk}
  Our definition of $\pi_n(X,x_0)$ is equivalent to the one
  given in \cite{Zem}, in which it is defined as the set of
  path components of 
  $\Loops^n(X,x_0)$, where
  \[
    \Loops(Y,y) = \Cinfty((\R,0,1),(Y,y,y))
  \]
  for any pointed diffeological space $(Y,y)$.
  On the other hand, our $\pi_n(X,x_0)$ is isomorphic to the
  set of path components
  of 
  $\Omega^n(X,x_0)$, where
  \[
    \Omega(Y,y) = \Cinfty((I,0,1),(Y,y,y)).
  \]
  But the inclusion $(I,0,1) \to (\R,0,1)$ is a homotopy
  equivalence because it has a homotopy inverse
  $\lambda \colon (\R,0,1) \to (I,0,1)$.
  Thus, by Proposition~\ref{prp:mapping space is homotopy
    invariant} we have $\Loops(Y,y) \simeq \Omega(Y,y)$,
  hence $\Loops^n(X,x_0) \simeq \Omega^n(X,x_0)$ for all
  $n \geq 0$.
  The situation is similar for the homotopy sets of pairs
  $\pi_n(X,A,x_0)$.
\end{rmk}
}
%
\section{Tame maps and approximate retractions}
In the case of topological spaces, the fact that $J^{n-1}$ is a
retract of $I^n$ is crucial for developing homotopy theory (e.g.\
homotopy exact sequence and homotopy extension property).
Unfortunately, this scenario does not work in the smooth case
because there is no smooth retraction $I^n \to J^{n-1}$.
Still, we can retrieve most of the ingredients of homotopy theory
by replacing retractions with a more relaxed notion of
\emph{approximate retractions}.



For every $1 \leq j \leq n$ and $\alpha \in \{0,1\}$, let
$\pi_j^{\alpha}$ denote the orthogonal projection of $I^n$ onto
its $(n-1)$-dimensional face
$\{(t_1,\cdots,t_n) \in I^n \mid t_j = \alpha\}$.

\begin{dfn}
  \label{dfn:definition of tame maps}
  (1) Let $K$ be a subset of $I^n$ and $0< \epsilon \leq 1/2$.  A
  smooth map $f \colon K \to X$ is said to be
  \emph{$\epsilon$-tame} if for every $P = (t_1,\dots,t_n) \in K$
  we have $f(P) = f(\pi_j^{\alpha}(P))$ whenever
  $\abs{t_j-\alpha} \leq \epsilon$ and $\pi_j^{\alpha}(P) \in K$
  hold.

  (2) A smooth homotopy $H \colon X \times I \to Y$ is said to be
  \emph{$\epsilon$-tame} if so is its adjoint
  $I \to \Cinfty(X,Y)$, that is, there hold $H(x,t) = H(x,0)$ for
  $0 \leq t \leq \epsilon$ and $H(x,t) = H(x,1)$ for
  $1-\epsilon \leq t \leq 1$.
\end{dfn}

%

Note that $\epsilon$-tameness implies $\sigma$-tameness for all
$\sigma < \epsilon$.
We use the abbreviation ``tame'' to mean $\epsilon$-tame for some
$\epsilon > 0$.

For $0 < \epsilon \leq 1/2$, let
$I^n(\epsilon) = [\epsilon,1-\epsilon]^n$ and call it the
\emph{$\epsilon$-chamber} of $I^n$.
More generally, if $K$ is a cubical subcomplex (i.e.\ a union of
faces) of $I^n$ then its $\epsilon$-chamber $K(\epsilon)$ is
defined to be the union of $\epsilon$-chambers of its maximal
faces.
Thus we have ${(\bI^n)(\epsilon)} = \bigcup F(\epsilon)$, where $F$
runs through the $(n-1)$-dimensional faces of $I^n$, and
$J^{n-1}(\epsilon) = {(\bI^n)(\epsilon)} \cap J^{n-1}$.

It is evident that the following holds.

\begin{lmm}
  \label{lmm:uniqueness criterion for tame maps}
  Let $f$ and $g$ be smooth maps from a cubical subcomplex $K$ of
  $I^n$ to a diffeological space $X$.  Suppose both $f$ and $g$
  are $\epsilon$-tame.  Then $f$ and $g$ coincide on $K$ if and
  only if they coincide on $K(\epsilon)$.
\end{lmm}

We show that any tame map defined on $J^{n-1}$ is extendable
over $I^n$.  
{ For this purpose, and to proceed further, we
  need a refinement of $\lambda$.  }

\begin{lmm}
  \label{lmm:modified smash function}
  Suppose $0 \leq \sigma < \tau \leq 1/2$.  Then there exists a
  non-decreasing smooth function $T_{\sigma,\tau} \colon \R \to I$
  satisfying the following conditions:
  \begin{enumerate}
  \item $T_{\sigma,\tau}(t) = 0$ for $t \leq \sigma$,
  \item $T_{\sigma,\tau}(t) = t$ for $\tau \leq t \leq 1 - \tau$,
  \item $T_{\sigma,\tau}(t) = 1$ for $1 - \sigma \leq t$, and
  \item $T_{\sigma,\tau}(1 - t) = 1 - T_{\sigma,\tau}(t)$
    for all $t$.
  \end{enumerate}
\end{lmm}

\begin{proof}
  For every $t \in \R$, put
  \[
    F(t) = \int_0^t \lambda\left(\frac{\tau x - \sigma}{\tau -
        \sigma}\right)dx +
    \frac{\tau+\sigma}{2\tau}\lambda\left(\frac{\tau t -
        \sigma}{\tau - \sigma}\right).
  \]
  Then $F \colon \R \to \R$ is a non-decreasing smooth function
  such that $F(t)$ has value $0$ for $t \leq \sigma/\tau$ and has
  value $t$ for $t \geq 1$.  Now, let us define a function
  $T_{\sigma,\tau} \colon \R \to I$ by putting
  $T_{\sigma,\tau}(t) = F(t/\tau)$ for $t \leq 1/2$, and
  $T_{\sigma,\tau}(t) = 1 - F((1-t)/\tau)$ for $1/2 \leq t$.
  As we have $T_{\sigma,\tau}(t) = t$ for $\tau \leq t \leq 1-\tau$,
  the function $T_{\sigma,\tau}$ is smooth all over $\R$ and satisfies
  the desired conditions.
\end{proof}
{ Note that $T_{0,1/2}$ differs from but can be
  used as an replacement of $\lambda$ thanks to its
  properties.  The graph of $T_{\sigma,\tau}$ looks as
  follows: }
{
\begin{center}
\begin{picture}(250,150)
\put(0,30){\vector(1,0){250}}
\put(60,0){\vector(0,1){140}}
\put(60,30){\dashbox{2}(100,100)}
\put(60,30){\dashbox{2}(30,30)}
\put(60,30){\dashbox{2}(70,70)}
\put(60,30){\dashbox{2}(90,100)}
\put(90,60){\line(1,1){40}}
\put(150,130){\line(1,0){70}}
\qbezier(130,100)(139,110)(143,120)
\qbezier(143,120)(145,129)(148,130)
\qbezier(70,30)(75,31)(77,40)
\qbezier(77,40)(85,58)(90,60)
\put(52,20){\footnotesize$0$}
\put(52,58){\footnotesize$\tau$}
\put(44,98){\footnotesize$1$-$\tau$}
\put(52,127){\footnotesize$1$}
\put(68,20){\footnotesize$\sigma$}
\put(87,20){\footnotesize$\tau$}
\put(123,20){\footnotesize$1$-$\tau$}
\put(142,20){\footnotesize$1$-$\sigma$}
\put(158,20){\footnotesize$1$}
\put(98,80){\footnotesize $id$}
\end{picture}
\end{center}
}
\begin{lmm}
 \label{lmm:taming of maps}
  Let $K$ be a cubical subcomplex of $I^n$.  Then for any smooth map
  $f \colon K \to X$ and $0 < \sigma < \epsilon \leq 1/2$, there
  exists a homotopy $f \simeq g$ relative to $K(\epsilon)$ such that
  $g$ is $\sigma$-tame.
  If $f$ is $\epsilon$-tame on a subcomplex $L$ of $K$ then the
  homotopy can be taken to be relative to $L \cup K(\epsilon)$.
\end{lmm}

\begin{proof}
  Let $g = f \circ T_{\sigma,\epsilon}^n|K \colon K \to X$.
  Then $g$ is $\sigma$-tame and there is a homotopy $f \simeq g$
  relative to $K(\epsilon)$ given by the map $K \times I \to X$
  which takes $(v,t) \in K \times I$ to
  $f((1-t)v + t T_{\sigma,\epsilon}^n(v))$.
  If $f$ is $\epsilon$-tame on $L$ then the homotopy $f \simeq g$
  is relative to $L \cup K(\epsilon)$ because $g$ coincides with
  $f$ on $L$.
\end{proof}

It follows, in particular, that any element of $\pi_n(X,A,x_0)$
is represented by a tame map $(I^n,\bI^n,J^{n-1}) \to (X,A,x_0)$.

A map $I^n \to J^{n-1}$ is called an \emph{$\epsilon$-approximate
  retraction} if it restricts to the identity on the
$\epsilon$-chamber $J^{n-1}(\epsilon)$.



{%
\begin{lmm}
  \label{lmm:approximate retraction exists}
  For any real numbers $\epsilon,\, \sigma$ such that
  $0 < \sigma < \epsilon < 1/2$, there exists an
  $\epsilon$-approximate retraction
  $R_{\sigma,\epsilon} \colon I^n \to J^{n-1}$ which is
  $\sigma$-tame.
\end{lmm}

\begin{proof} %
  For any $t = (t_1,\cdots,t_{n-1}) \in I^{n-1}$ and
  $u \in I$, put
  \begin{align*}
    &L(t)
      = {\prod}_{1 \leq k \leq n-1}
      T_{0,1/2}\left(\frac{t_k}{\sigma}\right)
      T_{0,1/2}\left(\frac{1-t_k}{\sigma}\right)
    \\
    &v(t,u)
      = L(t) + (1 - L(t))T_{\sigma,\epsilon}(u)
  \end{align*}
  Then $L(t)$ has value $1$ for $t \in I^{n-1}(\sigma)$ and
  $0$ for $t \in \bI^{n-1}$, and consequently, $v(t,u)$ has
  value $1$ for $(t,u) \in I^{n-1}(\sigma) \times I$, and
  $T_{\sigma,\epsilon}(u)$ for
  $(t,u) \in \bI^{n-1} \times I$.
  Clearly, the smooth map
  $R_{\sigma,\epsilon} \colon I^n \to J^{n-1}$ given by the
  formula
  \[
    R_{\sigma,\epsilon}(t,u) =
    (T_{\sigma,\epsilon}(t_1),\cdots,T_{\sigma,\epsilon}(t_{n-1}),
    v(t,u)).
  \]
  is $\sigma$-tame and restricts to the identity on
  $J^{n-1}(\epsilon)$.
\end{proof}

By combining this with Lemma~\ref{lmm:uniqueness criterion
  for tame maps}, we obtain the following.

\begin{prp}
  \label{prp:tame maps are extendable}
  Any $\epsilon$-tame map $f \colon J^{n-1} \to X$ can be
  extended to a $\sigma$-tame map $g \colon I^n \to X$ for
  any $\sigma < \epsilon$.
  %
\end{prp}

\begin{proof}
  Define $g \colon I^n \to X$ as the composition
  $f \circ R_{\sigma,\epsilon}$.  Then $g$ is $\sigma$-tame
  and coincides with $f$ on $J^{n-1}$ by
  Lemma~\ref{lmm:uniqueness criterion for tame maps}.
\end{proof}

The following is an immediate consequence of the proposition
above.  

\begin{thm}
  \label{thm:WHEP}
  Let $(X,A)$ be a pair of cubical subcomplexes of $I^m$,
  and $f \colon X \to Y$ be a tame map.
  Suppose there is a tame homotopy $h \colon A \times I \to Y$
  satisfying $h_0 = f|A$.  Then there exists a tame homotopy
  $H \colon X \times I \to Y$ satisfying $H_0 = f$ and
  $H|A \times I = h$.
\end{thm}

\begin{proof}
  $X \times I$ is obtained from
  $A \times I \cup X \times \{0\}$ by successively attaching
  cubes of the form $I^n$ along its subcomplex
  $L^{n-1} = \bI^{n-1} \times I \cup I^{n-1} \times \{0\}$.
  Thus, Proposition~\ref{prp:tame maps are extendable} (but
  with $J^{n-1}$ replaced by its copy $L^{n-1}$) enables us
  to extend
  $h \cup f \colon A \times I \cup X \times \{0\} \to Y$ to
  a tame homotopy $H \colon X \times I \to Y$.
\end{proof}
}%

For any pointed pair of diffeological spaces $(X,A,x_0)$, let
\[
  i_* \colon \pi_n(A,x_0) \to \pi_n(X,x_0), \quad %
  j_* \colon \pi_n(X,x_0) \to \pi_n(X,A,x_0)
\]
be the maps induced, respectively, by the inclusions
$i \colon (A,x_0) \to (X,x_0)$ and
$j \colon (X,x_0,x_0) \to (X,A,x_0)$, and let
\[
  \varDelta \colon \pi_n(X,A,x_0) \to \pi_{n-1}(A,x_0) %
  \quad (n \geq 1)
\]
be the map which takes the class of
$\phi \colon (I^n,\bI^n,J^{n-1}) \to (X,A,x_0)$ to the class of
its restriction
$\phi|I^{n-1} \colon (I^{n-1},\bI^{n-1}) \to (A,x_0)$.  Here, we
identify $I^{n-1}$ with $I^{n-1} \times \{0\} \subset I^n$.
Clearly, $\varDelta$ is a group homomorphism for $n \geq 2$.

Since any element of the homotopy group has a tame representative, we
can obtain the homotopy exact sequence by arguing as in the case of
topological spaces.  (Compare \cite[5.19]{Zem}.)

\begin{prp}
  \label{prp:homotopy exact sequence}
  Given a pointed pair of diffeological spaces $(X,A,x_0)$, there is
  an exact sequence of pointed sets
  \begin{multline*}
    \cdots \xrightarrow{} \pi_{n+1}(X,A,x_0) \xrightarrow{\varDelta}
    \pi_n(A,x_0) \xrightarrow{i_*} \pi_n(X,x_0) \xrightarrow{j_*}
    \pi_n(X,A,x_0) \xrightarrow{} \cdots
    \\
    \cdots \xrightarrow{} \pi_1(X,A,x_0) \xrightarrow{\varDelta}
    \pi_0(A,x_0) \xrightarrow{i_*} \pi_0(X,x_0).
  \end{multline*}
\end{prp}

{%
There is an alternative interpretation of $\pi_n(X,x_0)$ in
terms of basepointed maps $(\bI^{n+1},e) \to (X,x_0)$, where
$e = (1,\cdots,1) \in \bI^{n+1}$.

\begin{lmm}
  \label{lmm:another definition of homotopy groups}
  For every $n \geq 0$,
  there is a natural isomorphism
  \[
    \pi_n(X,x_0) \cong [\bI^{n+1},e \semicolon X,x_0].
  \]
\end{lmm}

\begin{proof}
  Consider the commutative diagram
  \[
    \xymatrix{%
      [I^n,\bI^n \semicolon X,x_0] & [\bI^{n+1},J^n
      \semicolon X,x_0] \ar[l]_-{i^*} \ar[r]^-{j^*} &
      [\bI^{n+1},e \semicolon X,x_0]
      \\
      [I^n/\bI^n,* \semicolon X,x_0] \ar[u]_-{\cong} &
      [\bI^{n+1}/J^n,* \semicolon X,x_0] \ar[l]_-{\cong}
      \ar[u]_-{\cong} }%
  \]
  induced by the evident inclusions and projections.
  One easily observes that the vertical arrows are
  isomorphisms by the property of subductions, and the lower
  horizontal arrow is an isomorphism induced by a
  diffeomorphism.  It follows by the commutativity of the
  left hand square that $i^*$ is an isomorphism.  
  Hence,
  to prove the lemma\textcolor{red}{,}
  it suffice to show that $j^*$ is also an isomorphism.

  To see that $j_*$ is surjective, take a tame map
  $f \colon (\bI^{n+1},e) \to (X,x_0)$.  Since
  $I^{n-1} \times \{1\}$ is a deformation retract of $J^n$
  and is contractible to $e$, there exists a tame
  contracting homotopy $r \colon J^n \times I \to J^n$ of
  $J^n$ onto $e$.
  By applying Theorem~\ref{thm:WHEP} to the map $f$ and the
  homotopy $h = (f|J^n) \circ r \colon J^n \times I \to X$,
  we obtain a homotopy $f \simeq g$ relative to $e$ such
  that $g(J^n) = x_0$.  Hence we have
  $[f] = j^*([g]) \in [\bI^{n+1},e \semicolon X,x_0]$,
  meaning that $j^*$ is surjective.

  Injectivity of $j^*$ is proved as follows.  Let $f$ and
  $g$ be two tame maps $(\bI^{n+1},J^n) \to (X,x_0)$ such
  that $j^*([f]) = j^*([g])$.  Then there are a tame
  homotopy $H \colon (\bI^{n+1},e) \times I \to (X,x_0)$
  between $f \circ j$ and $g \circ j$, and a tame homotopy
  (of homotopies) from $H|(J^n,e) \times I$ to the trivial
  homotopy
  $c_{x_0} \colon (J^n,e) \times I \to (x_0,x_0) \subset
  (X,x_0)$ induced by $r$.  Thus, by Theorem~\ref{thm:WHEP}
  again, $H$ is homotopic to $H'$ such that
  $H'|(J^n,e) \times I = c_{x_0}$, implying that
  $H'_0 \simeq H'_1 \colon (\bI^{n+1},J^n) \to (X,x_0)$ and
  hence $[f] = [H'_0] = [H'_1] = [g]$.
\end{proof}
}%

Suppose $f \colon (\bI^{n+1},e) \to (X,x_0)$ is a tame representative
of an element of $\pi_n(X,x_0)$, and $\ell \colon I \to X$ a tame path
from $x_0$ to $x_1$.
Then, by applying Proposition~\ref{prp:tame maps are extendable} to the tame homotopy
$e \times I \to X$ given by $l$, we obtain a homotopy $f \simeq g$
such that $g(e) = x_1$ holds.
Thus we can construct
\[
  \ell_{\sharp} \colon \pi_n(X,x_0) \to \pi_n(X,x_1)
\]
to be the map which takes $[f] \in \pi_n(X,x_0)$ to the class
$[g] \in \pi_n(X,x_1)$.

We leave it to the reader to verify the following.

\begin{prp}
  \label{prp:invariance under basepoint change}
  To every tame path $\ell \colon I \to X$ joining $x_0$ to $x_1$,
  there attached a group isomorphism
  $\ell_{\sharp} \colon \pi_n(X,A,x_0) \to \pi_n(X,x_1)$.
  If $\ell' \colon I \to X$ is another tame path joining $x_1$ to
  $x_2$ then we have
  $(\ell * \ell')_{\sharp} = \ell'_{\sharp} \circ \ell_{\sharp}$.
\end{prp}
%
%
%
%
%
%
%
%
%
\section{Fibrations}
In this section, we introduce the notion of smooth fibrations by
mimicking Serre fibrations in $\Top$.
%

\begin{dfn}
  \label{dfn:admissibility}
  Let $K$ be a subset of $I^n$ and let $0 < \epsilon \leq 1/2$.
  Then a smooth map $f \colon K \to X$ is said to be
  \emph{$\epsilon$-admissible} if its restriction $f|K \cap F$ is
  $\epsilon^{\dim{F}}$-tame for every positive dimensional face
  $F$ of $I^n$.
\end{dfn}

Clearly, $\epsilon$-tameness implies $\epsilon$-admissibility
and, conversely, $\epsilon$-admissibility implies
$\epsilon^d$-tameness for some $d \leq n$.
But admissibility has an advantage over tameness as exhibited by
the next proposition and its corollary.

\begin{prp}
  \label{prp:admissible replacement}
  Let $K$ be a cubical subcomplex of $I^n$, and
  $f \colon K \to X$ be a smooth map.  Suppose $f$ is
  $\epsilon$-admissible on a cubical subcomplex $L$ of $K$.  Then
  there is a homotopy $f \simeq g$ relative to $L$ such that $g$
  is $\epsilon$-admissible.
\end{prp}

\begin{proof}
  It is easy to construct a homotopy $f \simeq f' \rel{L}$ such
  that $f'$ is $\sigma$-tame for $\sigma < \epsilon^{\dim{L}}$
  (cf.\ Lemma~\ref{lmm:taming of maps}).  Hence we may assume
  from the beginning that $f$ is a tame map.
  For $0 \leq j \leq \dim K$, let $\bar{K}^j = L \cup K^j$ be the
  union of $L$ and the $j$-skeleton of $K$.
  Starting from the constant homotopy
  $\tilde{h}^0 \colon \bar{K}^0 \times I \to X$, we inductively
  construct a tame homotopy
  $\tilde{h}^j \colon \bar{K}^j \times I \to X$ from
  $f|\bar{K}^j$ to an $\epsilon$-admissible map $g^j$ relative to
  $L$.

  Suppose $\tilde{h}^{j-1}$ exists.  Let $F$ be a $j$-dimensional
  face not contained in $L$.  As
  $\partial F \subset \bar{K}^{j-1}$, there is a tame map
  \( h^F \colon (\partial F \times I) \cup (F \times \{0\}) \to
  X, \) which takes $(t,u)$ to $\tilde{h}^{j-1}(t,u)$ if
  $t \in \partial F$ and to $f(t)$ if $u = 0$.
  But as $g^{j-1}|\partial F$ is $\epsilon^{j-1}$-tame and
  $(\partial F \times I) \cup (F \times \{0\})$ is linearly
  diffeomorphic to $J^{j-1}$, we can apply
  Proposition~\ref{prp:tame maps are extendable} with
  sufficiently small $\sigma$ and $\sigma' = \epsilon^{j}$ to
  obtain a tame extension $\tilde{h}^F \colon F \times I \to X$
  such that $g^F = \tilde{h}^F|F \times \{1\} \colon F \to X$ is
  $\epsilon$-admissible.
  Thus, if we define
  $\tilde{h}^j \colon \bar{K}^j \times I \to X$ to be the union
  $\bigcup_F \tilde{h}^F$, where $F$ runs through $j$-dimensional
  faces of $K$ not contained in $L$, then $\tilde{h}^j$ gives a
  tame homotopy $f|\bar{K}^j \simeq g^j \rel{L}$ such that
  $g^j = \bigcup_F g^F$ is $\epsilon$-admissible, proving the
  induction step.
\end{proof}

This, together with Proposition~\ref{prp:tame maps are
  extendable}, implies the following.

\begin{crl}
  \label{crl:admissible maps are extendable}
  Any $\epsilon$-admissible map $f \colon J^{n-1} \to X$ can be
  extended to an $\epsilon$-admissible map $I^n \to X$.
\end{crl}
Consider the following classes of inclusions:
\[
  \mathcal{I} = \{i_n \colon \partial I^{n} \to I^{n} \mid n \geq
  0 \}, \quad \mathcal{J} = \{j_n \colon J^{n-1} \to I^{n} \mid n
  \geq 1\}.
\]

\begin{dfn}
  \label{dfn:K-fibration}
  Let $\K$ be either $\I$ or $\J$.  A smooth map
  $p \colon X \to Y$ is called a \emph{$\K$-fibration} if for
  every member $K^{n-1} \to I^n$ of $\K$ (i.e.\ $K^{n-1}=\partial I^{n}$ or $J^{n-1}$)
  and every pair of $\epsilon$-admissible maps
  $f \colon K^{n-1} \to X$ and $g \colon I^n \to Y$ satisfying
  $p \circ f = g|K^{n-1}$, there exists an
  $\epsilon$-admissible map $h \colon I^n \to X$ which makes the
  two triangles in the diagram below commutative:
  \begin{equation}
    \label{eqn:RLP for fibration}
    \vcenter{%
      \xymatrix{%
        K^{n-1} \ar[d] \ar[r]^f & X \ar[d]^p
        \\
        I^n \ar[r]^g \ar@{.>}[ru]^{h} & Y.} }%
  \end{equation}
\end{dfn}
In particular,
$\J$-fibrations are analogy of Serre fibrations in $\Top$.

\begin{ex}\label{examples of fibrations}
(1) It follows by Corollary \ref{crl:admissible maps are extendable} that for any diffeological space $X$ the constant map $X \to \ast$ is a $\J$-fibration.

(2)
If $p \colon E \to B$ is a diffeological fiber bundle with fiber $F$ 
then its pullback by a smooth map from $I^{n}$ to $B$ is 
trivial (cf.~{\cite[8.19,\ Lemma 2]{Zem}}).
But (1) implies that a trivial fiber bundle is a $\J$-fibration,
hence so is $p$.

(3)
Given a diffeological space $X$ with basepoint $x_{0}$,
let $P(X,x_{0})$ be the subset of $C^{\infty}((I,\{1\}),(X,x_{0}))$ consisting of tame paths $l \colon I \to X$ satisfying $l(1)=x_{0}$.
Then the map $p \colon P(X,x_{0}) \to X$,
which takes a path $l$ to its initial point $l(0)$,
is a $\J$-fibration.
To see this,
let $u \colon J^{n-1} \to P(X,x_{0})$ and $v \colon I^{n} \to X$ be $\epsilon$-admissible maps satisfying $p \circ u=v|J^{n-1}$.
Let $K=J^{n-1} \times I \cup I^{n} \times \{0,1\}$ and $u^{\prime} \colon K \to X$ be a tame map which takes $(t,s) \in K$ to $u(t)(s)$ 
if $t \in J^{n-1}$,
to $v(t)$ if $s =0$,
and to $x_{0}$ if $s=1$.
To obtain an $\epsilon$-admissible lift $I^{n} \to P(X,x_{0})$ of $(u,v)$,
it suffices to extend $u^{\prime}$ to a tame map $\tilde{u} \colon I^{n} \times I \to X$ which is $\epsilon$-admissible with respect to the 
first $n$ coordinates.
We accomplish this by extending $u^{\prime}$ in several steps.
Let $A= I^{n}(\epsilon^{n}) \times I,\ B=I^{n-1}(\epsilon^{n}) \times [0,1-\epsilon^{n}] \times I$,
and $C=I^{n} \times I- {\rm Int}B$.
As we have $C=K \cup (\bar{J}^{n-1} \times I)$,
where $\bar{J}^{n-1}$ is the closure of the $\epsilon^{n}$-neighborhood of $J^{n-1}$,
and $v$ is $\epsilon^{n}$-tame,
$u^{\prime}$ can be extended to a tame map $\tilde{u}^{\prime} \colon C \to X$ in an evident manner.
But as $(A,A \cap C) \cong (I^{n+1}, J^{n})$ and $\tilde{u}^{\prime}$ is tame on $A \cap C$,
there exists an extension $\tilde{u}^{\prime \prime} \colon A \to X$ of $\tilde{u}^{\prime}|A \cap C$ having enough tameness on 
$A \cap (I^{n-1} \times \{\epsilon^{n}\} \times I)$ (cf.~Proposition \ref{prp:tame maps are extendable}).
It is now clear that $\tilde{u}^{\prime \prime}$ can be extended trivially to $\tilde{u}^{\prime \prime \prime} \colon B \to X$,
and the resulting map $\tilde{u}=\tilde{u}^{\prime} \cup \tilde{u}^{\prime \prime \prime} \colon I^{n+1}=C \cup B \to X$ extends $u^{\prime}$ 
and is $\epsilon$-admissible with respect to the first $n$ coordinates as desired.
\end{ex}

We say that a smooth map $p \colon X \to Y$ is a weak homotopy
equivalence if the induced map
$p_{\ast} \colon \pi_{n}(X,x) \to \pi_{n}(Y,p(x))$ is a bijection
for every $x \in X$ and $n \geq 0$.

\begin{prp}
  \label{prp:trivial J-fibration is I-fibration}
  A smooth map $p \colon X \to Y$ is an $\I$-fibration if and
  only if it is a $\J$-fibration and a weak homotopy equivalence.
\end{prp}

\begin{proof}
  Suppose $p \colon X \to Y$ is an $\I$-fibration.  We show that
  $p$ is a weak equivalence, that is, the induced map
  $p_{\ast}\colon \pi_n(X,x) \to \pi_n(Y,p(x))$ is bijective for
  every $n \geq 0$ and $x \in X$.  Let
  $\gamma \colon (I^n, \bI^n) \to (Y,p(x))$ be a tame map,
  and let $c_x \colon \bI^n \to X$ be the constant map with value
  $x \in X$.  Then we have a commutative square
  \[
    \xymatrix@C=32pt{%
      \bI^n \ar[r]^-{c_x} \ar[d]^-{i_n} & X \ar[d]^-{p}
      \\
      I^n \ar[r]^-{\gamma} & Y.
    }%
  \]
  Since $c_x$ and $\gamma$ are $\epsilon$-admissible for some
  $\epsilon > 0$, there is a lift
  $\tilde{\gamma} \colon I^n \to X$ satisfying
  $\tilde{\gamma} \circ i_n = c_x$ and
  $p \circ \tilde{\gamma} = \gamma$.  Thus we have
  $p_{\ast}([ \tilde{\gamma}]) = [\gamma]$, implying that
  $p_{\ast}$ is a surjection.
  To see that $p_*$ is injective, let $\gamma_0$ and $\gamma_1$ be
  tame maps $(I^n,\bI^n) \to (X,x)$ such that
  $p_{\ast}([\gamma_0]) = p_{\ast}([\gamma_1])$ holds in
  $\pi_n(Y,p(x))$.
  Then there exists a tame homotopy $H \colon I^n \times I \to Y$
  between $p \circ \gamma_0$ and $p \circ \gamma_1$ relative to
  $\bI^n$.
  Let $\gamma \colon \bI^{n+1} \to X$ be a tame map which takes
  $(t,s)$ to $\gamma_s(t)$ if
  $(t,s) \in I^n \times \{0,1\}$, and to $x$ if
  $(t,s) \in \bI^n \times I$.  Then we have a commutative square
  \[
    \xymatrix{%
      \bI^{n+1} \ar[r]^-{\gamma} \ar[d]^-{i_{n+1}}& X \ar[d]^-{p}
      \\
      I^n \times I \ar[r]^-{H} & Y.
    }%
  \]
  Hence there exists a lift $\tilde{H} \colon I^n \times I \to X$
  which gives a homotopy $\gamma_0 \simeq \gamma_1$.  Thus we have
  $[\gamma_0] = [\gamma_1]$, showing that $p_{\ast}$ is injective.

  We now show that $p$ is a $\J$-fibration.  Let
  $0 < \epsilon \leq 1/2$ and take $\epsilon$-admissible maps
  $f \colon J^{n-1} \to X$ and $g \colon I^n \to Y$ satisfying
  $p \circ f = g \circ i_n$.  Then we have a commutative square
  \[
    \xymatrix@C=60pt{%
      \bI^{n-1} \times \{0\} \ar[d]^{i_{n-1}} \ar[r]^-{f|\bI^{n-1}
        \times \{0\}} & X \ar[d]^p
      \\
      I^{n-1} \times \{0\} \ar[r]^-{g|I^{n-1}\times \{0\}} & Y.  }%
  \]
  Since $f|\bI^{n-1} \times \{0\}$ and $g|I^{n-1}\times \{0\}$ are $\epsilon$-admissible,
  there is an $\epsilon$-admissible lift $\tilde{g} \colon I^{n-1} \times \{0\} \to X$, and consequently we
  can define $\tilde{f} \colon \bI^n \to X$ to be the union
  $f \cup \tilde{g} \colon J^{n-1} \cup I^{n-1} \times \{0\} \to X$.
  Clearly, $\tilde{f}$ is $\epsilon$-admissible, and hence there
  exists an $\epsilon$-admissible lift $G \colon I^n \to X$ satisfying
  $p \circ G = g$ and $G \circ i_n = \tilde{f}$.  But this means
  $G|J^{n-1} = f$, implying that $p$ is a $\J$-fibration.

  Conversely, suppose $p \colon X \to Y$ is a $\J$-fibration and
  a weak homotopy equivalence.  Let $f \colon \bI^n \to X$ and
  $g \colon I^n \to Y$ be $\epsilon$-admissible maps satisfying
  $p \circ f=g \circ i_{n}$.  We need to show that there is an
  $\epsilon$-admissible lift $G \colon I^n \to X$ satisfying
  $p \circ G = g$ and $G \circ i_n = f$.
  %
  Let $e = (1,\cdots,1) \in \partial I^n$ and $x = f(e) \in X$. 
  Since
  $p \circ f = g|\bI^n$ is null homotopic and $p$ is a weak
  equivalence, there exists,
  by Lemma~\ref{lmm:another definition of homotopy groups},
  a tame homotopy
  $F \colon \bI^n \times I \to X$ from $f$ to the constant map.
  %
  %
  Let us define $H \colon \partial I^{n} \times I \cup I^{n} \times \{0\} \to Y$ by
  \[
  H(t,s) =
  \begin{cases}
    g(t), & (t,s) \in I^n \times \{0\}
    \\
    p(F(t,s)), & (t,s) \in \bI^n \times I.
  \end{cases}
  \]
  Since $H$ is $\epsilon$-admissible,
  it can be extended by
  Corollary \ref{crl:admissible maps are extendable}
  to an $\epsilon$-admissible homotopy $H' \colon I^n \times I \to Y$ from $g$ to
  $\gamma' \colon (I^n,\bI^n) \to (Y,p(x))$.
  But as $p$ is a weak equivalence, there exist a tame map
  $\gamma \colon (I^n, \bI^n) \to (X,x)$ and an $\epsilon$-admissible homotopy
  $H'' \colon I^n \times I \to Y$ from $\gamma^{\prime}$ to $p \circ \gamma$ relative to $\bI^{n}$.
  %
  %
  By tameness,
  we can define smooth maps $F^{\prime} \colon J^n \to X$ and
  $G^{\prime} \colon I^n \times I \to Y$ by the formula,
  \begin{align*}
    F^{\prime}(t,s)
    &=
      \begin{cases}
        F(t, 2s), & (t,s) \in\bI^n \times [0,1/2]
        \\
        x, & (t,s) \in \bI^n \times [1/2,1]
        \\
        \gamma(t), & (t,s) \in I^n \times \{1\},
      \end{cases}
    \\                 
    G'(t,s)
    &=
      \begin{cases}
        H'(t,2s)), & 0 \leq s \leq 1/2
        \\
        H''(t, 2s-1), & 1/2 \leq s \leq 1.
      \end{cases}
  \end{align*}
   Let $F^{\prime \prime} \colon J^{n} \to X$ and $G^{\prime \prime} \colon I^{n} \times I \to Y$ be $\epsilon$-admissible maps defined by $F^{\prime \prime}(t,s)=F^{\prime}(t, \lambda_{\epsilon}(s))$
   and $G^{\prime \prime}(t,s)=G^{\prime}(t,\lambda_{\epsilon}(s))$.
Then there exists an $\epsilon$-admissible lift $\tilde{G} \colon I^{n} \times I \to X$ satisfying $\tilde{G}|J^{n}=F^{\prime \prime}$ and 
$p \circ \tilde{G}=G^{\prime \prime}$,
since $p \circ F^{\prime \prime}=G^{\prime \prime}|J^{n}$ and $p$ is a $\J$-fibration.
Hence we have an $\epsilon$-admissible lift $G=\tilde{G}|I^{n} \times \{0\} \colon I^{n} \to X$ satisfying $G \circ i_{n}=f$ and $p \circ G=g$.
  %
\end{proof}
%
%
%
%
%
%
%
%
%
%
%
\section{Model category of diffeological
  spaces}\label{sec:Quillen model category}
In this section we shall show that the category $\Diff$ has a
model structure by arguing as in the proof of \cite[Proposition
8.3]{Spa}.

\begin{dfn}\label{dfn:cofibrations}
  Let $\K$ be either $\I$ or $\J$.  A smooth map
  $i \colon X \to Y$ is called a \emph{$\K$-cofibration} if it
  has the left lifting property with respect to $\K$-fibrations,
  that is, for every commutative square
  \begin{equation*}
    \label{eqn:commutative square}
    \vcenter{%
      \xymatrix@C=32pt{%
        X \ar[d]^i \ar[r] & E \ar[d]^p
        \\
        Y \ar@{.>}[ur] \ar[r] & B }}%
  \end{equation*}
  such that $p \colon E \to B$ is a $\K$-fibration, there exists
  a lift $Y \to E$ making the two triangles commutative.
\end{dfn}

\begin{thm}\label{thm:model structure}
  The category $\Diff$ has a structure of a model category by
  defining a smooth map $h \colon X \to Y$ to be
  \begin{enumerate}
  \item a weak equivalence if $h$ is a weak homotopy equivalence,
  \item a fibration if $h$ is a $\J$-fibration, and
  \item a cofibration if $h$ is an $\I$-cofibration.
  \end{enumerate}
\end{thm}

Observe that every diffeological space $X$ is {\it fibrant} in sense that the constant map $X \to \ast$ is a $\J$-fibration 
(cf.~Example \ref{examples of fibrations}).

We prove Theorem~\ref{thm:model structure} by verifying the
axioms below (cf.~\cite{Spa}).  A fibration or a cofibration
is called to be \emph{trivial} if it is a weak homotopy
equivalence.

\begin{itemize}
\item[\bf MC1] Finite limits and colimits exist.
\item[\bf MC2] If $f$ and $g$ are maps such that $g \circ f$ is
  defined and if two of the three maps $f$, $g$, $g \circ f$ are weak
  equivalences, then so is the third.
\item[\bf MC3] If $f$ is a retract of $g$ and $g$ is a fibration,
  cofibration, or a weak equivalence, then so is $f$.
\item[\bf MC4] Given a commutative square of the form
  \begin{equation*}
    \vcenter{%
      \xymatrix@C=32pt{%
        A \ar[r] \ar[d]^-{i} & X \ar[d]^-{p}
        \\
        B \ar[r] \ar@{.>}[ru] & Y }}%
  \end{equation*}
  the dotted arrow exists so as to make the two triangles commutative
  if either (i) $i$ is a cofibration and $p$ is a trivial fibration,
  or (ii) $i$ is a trivial cofibration and $p$ is a fibration.
\item[\bf MC5] Any map $f$ can be factored in two ways: (i)
  $f = p \circ i$, where $i$ is a cofibration and $p$ is a trivial
  fibration, and (ii) $f = p \circ i$, where $i$ is a trivial
  cofibration and $p$ is a fibration.
\end{itemize}

Axiom \textbf{MC1} follows from the fact that $\Diff$ has small
limits and colimits, 
and \textbf{MC2} follows from the functoriality of induced maps
combined with the change of basepoint homomorphism
(Proposition~\ref{prp:invariance under basepoint change}).  Axiom
\textbf{MC3} is straightforward from the definitions (cf.\
\cite[8.10]{Spa}).  In order to verify \textbf{MC4} and
\textbf{MC5}, we need several lemmas and propositions.

By Proposition~\ref{prp:trivial J-fibration is I-fibration}, all
cofibrations have the left lifting property with respect to
trivial $\J$-fibrations.  Hence we have the following.

\begin{crl}
  \label{crl:MC4 (ii)}
  Axiom\/ {\bf MC4} holds under the condition {\rm (i).}
\end{crl}

The rest of the axioms follow from the theorem below, whose proof
is deferred until the next section.

\begin{thm}
  \label{thm:factorization of a smooth map}
  Let $\K$ be either $\I$ or $\J$.  Then any smooth map
  $f \colon X \to Y$ can be factorized as a composition
  \[
    X \xrightarrow{i_{\infty}} G^{\infty}(\K,f)
    \xrightarrow{p_{\infty}} Y
  \]
  such that $i_{\infty}$ is a $\K$-cofibration and $p_{\infty}$ is
  a $\K$-fibration.
  Moreover, $i_{\infty}$ can be taken as a trivial cofibration
  when $\K = \J$.
\end{thm}

Since $\I$-fibration is a trivial fibration by
Proposition~\ref{prp:trivial J-fibration is I-fibration}, the
factorization
\[
  X \xrightarrow{i_{\infty}} G^{\infty}(\I,f)
  \xrightarrow{p_{\infty}} Y
\]
implies \textbf{MC5} (1), while on the other hand,
\[
  X \xrightarrow{i_{\infty}} G^{\infty}(\J,f)
  \xrightarrow{p_{\infty}} Y
\]
implies \textbf{MC5} (2).  Finally, we prove \textbf{MC4}~(ii),
that is,

\begin{prp}
  Every trivial cofibration has the left lifting property with respect
  to fibrations.
\end{prp}

\begin{proof}
  Suppose $i \colon X \to Y$ is a trivial cofibration and
  $p \colon A \to B$ a fibration.  Let $f \colon X \to A$ and
  $g \colon Y \to B$ be smooth maps such that
  $p \circ f = g \circ i$ holds.
  Let us take the factorization
  $i = p_{\infty} \circ i_{\infty} \colon X \to G^{\infty}(\J,f)
  \to Y$, where $i_{\infty}$ is a trivial cofibration and
  $p_{\infty}$ is a fibration.  Because $i$ and $i_{\infty}$ are
  weak equivalences, $p_{\infty}$ is a weak equivalence, and
  hence a trivial fibration.

  Now, consider the commutative square
  \[
    \xymatrix{%
      X \ar[d]^-{i} \ar[r]^-{i_{\infty}} & G^{\infty}(\J,i)
      \ar[d]^-{p_{\infty}}
      \\
      Y \ar@{=}[r] \ar@{.>}[ru] & Y.
    }%
  \]
  As $i$ is a cofibration, there exists by MC4~(i) a lift
  $h \colon Y \to G^{\infty}(\J,i)$ such that
  $p_{\infty} \circ h =1$ and $ h\circ i = i_{\infty}$.
  Hence we obtain a commutative diagram
  \[
    \xymatrix{%
      X \ar[d]^-{i} \ar@{=}[r] & X \ar@{=}[r] \ar[d]^-{i_{\infty}} & X
      \ar[d]^-{i} \ar[r]^-{f} & A \ar[d]^-{p}
      \\
      Y \ar[r]^-{h} & G^{\infty}(\J,i) \ar[r]^-{p_{\infty}}
      \ar@{.>}[rru] & Y \ar[r]^-{g} & B.
    }%
  \]
  As $i_{\infty}$ is a $\J$-cofibration, there exists a lift
  $g^{\prime} \colon G^{\infty}(\J,i) \to A$ making the diagram
  commutative.  Thus we obtain a desired lift
  $g' \circ h \colon Y \to A$.
\end{proof}

This completes the proof of Theorem \ref{thm:model structure}.

\section{Infinite gluing construction}
We prove Theorem~\ref{thm:factorization of a smooth map} by applying infinite gluing construction to define a factorization
$X \to G^{\infty}(\K,f) \to Y$ for $\K=\I$ and $\J$.

For $0<\epsilon < \tau \leq 1/2$,
let $\tilde{I}^{n}_{\epsilon, \tau}$ be the $n$-cube equipped with the diffeology generated by the smooth map 
$T^{n}_{\epsilon, \tau} \colon \mathbf{R}^{n} \to I^{n}$ (cf.~Lemma \ref{lmm:modified smash function}).
By the definition,
$T^{n}_{\epsilon, \tau}$ restricts to a subduction $I^{n} \to \tilde{I}^{n}_{\epsilon, \tau}$.
\begin{lmm}\label{lmm:property of tame}
For any $\epsilon$-tame map $f \colon I^{n} \to X$,
there exists a smooth map $\tilde{f} \colon \tilde{I}^{n}_{\epsilon, \tau} \to X$ satisfying $f=\tilde{f} \circ T^{n}_{\epsilon,\tau}$.
\end{lmm}
\begin{proof}
Since $T^{n}_{\epsilon, \tau}$ restricts to bijection $[\epsilon, 1- \epsilon]^{n} \to \tilde{I}^{n}_{\epsilon,\tau}$,
and since $f$ is $\epsilon$-tame,
there is a well defined map $\tilde{f}=f \circ (T^{n}_{\epsilon, \tau})^{-1} \colon \tilde{I}^{n}_{\epsilon, \tau} \to X$ which satisfies 
$f=\tilde{f} \circ T^{n}_{\epsilon, \tau}$.
But as $T^{n}_{\epsilon,\tau}$ is a subduction,
$\tilde{f}$ is smooth by \cite[1.51]{Zem}.
\end{proof}
\begin{prp}\label{prp:homotopy equivalent of cubes}
The map $T^{n}_{\epsilon,\tau} \colon I^{n} \to \tilde{I}^{n}_{\epsilon,\tau}$ gives a homotopy inverse to the inclusion 
$1_{\epsilon,\tau} \colon \tilde{I}^{n}_{\epsilon,\tau} \to I^{n}$.
\end{prp}
\begin{proof}
Define $F \colon I^{n} \times I \to I^{n}$ by $F(t,u)=(1-u)t+uT^{n}_{\epsilon,\tau}(t)$.
Then $F$ gives a homotopy $1 \simeq 1_{\epsilon,\tau} \circ T^{n}_{\epsilon,\tau}$.
On the other hand,
if we define $G \colon \tilde{I}_{\epsilon,\tau} \times I \to \tilde{I}_{\epsilon,\tau}$ by 
$G(t,u)=T^{n}_{\epsilon,\tau}(F((T^{n}_{\epsilon,\tau})^{-1}(t),u))$,
then $G$ is smooth because its composition with the subduction $T^{n}_{\epsilon,\tau} \times 1$ is a smooth map 
$T^{n}_{\epsilon,\tau} \circ F$,
and gives a homotopy $1 \simeq T^{n}_{\epsilon,\tau} \circ 1_{\epsilon,\tau}$.
\end{proof}
For $0 < \delta < 1/2$,
let $J^{n-1}_{\delta}=\partial I^{n}\setminus
(\delta,1-\delta)^{n-1} \times \{0\}$.
\begin{lmm}
\label{lmm:equivalent of fibrations}
{%
  \emph{(1)} Any $\epsilon$-admissible map
  $f \colon J^{n-1} \to X$ can be extended to an
  $\epsilon$-admissible map
  $f_{\epsilon} \colon J^{n-1}_{\epsilon^{n-1}} \to X$.

  \emph{(2)} For any smooth map $p \colon X \to Y$ the
  following conditions are equivalent with each other.
  \begin{enumerate}
  \item[\rm (a)] $p$ is a $\J$-fibration.
  \item[\rm (b)] For every pair of $\epsilon$-admissible maps
    $f \colon J^{n-1}_{\epsilon^{n-1}} \to X$ and
    $g \colon I^n \to Y$ satisfying
    $p \circ f = g|J^{n-1}_{\epsilon^{n-1}}$, there exists
    an $\epsilon$-admissible map $h \colon I^n \to X$ such
    that $h|J^{n-1}_{\epsilon^{n-1}} = f$ and
    $p \circ h = g$ hold.
  \end{enumerate}
}%
\end{lmm}


\begin{proof}
  {%
    (1) \ Since $f$ is $\epsilon^{n-2}$-tame on
    \[
    J^{n-1}_{\epsilon^{n-1}} \cap I^{n-1} \times
    \{0\}=I^{n-1}\setminus(\epsilon^{n-1}, 1-\epsilon^{n-1})^{n-1},
    \]
    we can extend $f$
    to an $\epsilon^{n-1}$-admissible map
    $f_{\epsilon} \colon J^{n-1}_{\epsilon^{n-1}} \to X$ by
    assigning
    \[
      f_{\epsilon}(t,0) =
      f(T^{n-1}_{\epsilon^{n-1},\epsilon^{n-2}}(t),0) \ \
      \text{for} \ \ (t,0)\in J^{n-1}_{\epsilon^{n-1}} \cap
      I^{n-1} \times \{0\}.
    \]

    (2) \
  }%
  Suppose $p \colon X \to Y$ is a $\J$-fibration.  Let
  $f \colon J^{n-1}_{\epsilon^{n-1}} \to X$ and
  $g \colon I^{n} \to Y$ be $\epsilon$-admissible maps satisfying
  $p \circ f = g|J^{n-1}_{\epsilon^{n-1}}$.  Then we have
  $p \circ f|J^{n-1} = g|J^{n-1}$, and hence there exists an
  $\epsilon$-admissible map $h \colon I^{n} \to X$ satisfying
  $h|J^{n-1} = f|J^{n-1}$ and $p \circ h = g$.  But as $h$ is
  $\epsilon^{n-1}$-tame on $I^{n-1} \times \{0\}$, it must
  coincides with $f$ on
  $J^{n-1}_{\epsilon^{n-1}} \cap I^{n-1} \times \{0\}$.
  Thus we have $h|J^{n-1}_{\epsilon^{n-1}} = f$, implying that $p$
  satisfies the condition { (b)}.

  Conversely, suppose $p$ satisfies { (b)}.  Let
  $f \colon J^{n-1} \to X$ and $g \colon I^{n} \to Y$ be
  $\epsilon$-admissible maps satisfying $p \circ f = g|J^{n-1}$,
  {%
  and let
  $f_{\epsilon} \colon J^{n-1}_{\epsilon^{n-1}} \to X$ be an
  $\epsilon^{n-1}$-admissible extension of $f$ given by (1).
  Then we have
  $p \circ f_{\epsilon}=g|J^{n-1}_{\epsilon^{n-1}}$
  }%
  because $f_{\epsilon}$ and $g$ are
  $\epsilon^{n-1}$-tame on $J^{n-1}_{\epsilon^{n-1}} \cap I^{n-1} \times \{0\}$.
  Thus,
  there exists by { (b)} an $\epsilon$-admissible map $h \colon I^{n} \to X$ satisfying $h|J^{n-1}=f$ and $p \circ h=g$,
  meaning that $p$ is a $\J$-fibration.
\end{proof}
{%
In the sequel, we denote
$\tilde{I}^n_{\epsilon^n} =
\tilde{I}_{\epsilon^{n},\epsilon^{n-1}}^{\,n}$,
$T^{n}_{\epsilon^{n}}=T^{n}_{\epsilon^{n},\epsilon^{n-1}} \colon
I^{n} \to \tilde{I}^{n}_{\epsilon^{n}}$, and by
$K^{n-1}_{\epsilon}$ either $\bI^{n}$ or
$J^{n-1}_{\epsilon^{n-1}}$ according as $\K$ is $\I$ or
$\J$.
For any smooth map
}%
$\phi_{\epsilon} \colon K^{n-1}_{\epsilon} \to X$, let
$X \cup_{\phi_{\epsilon}} \tilde{I}^{n}_{\epsilon^{n}}$
denote the adjunction space given by a pushout square:
%
\[
  \xymatrix{%
    K^{n-1}_{\epsilon}
    \ar[r]^-{\phi_{\epsilon}} 
    \ar[d]^-{{T^{n}_{\epsilon^{n}}}} 
    & 
    X 
    \ar[d]^-{i}
    \\
    \tilde{I}^{n}_{\epsilon^{n}}
    \ar[r]^-{\Phi} 
    & 
    X \cup_{\phi_{\epsilon}} \tilde{I}^{n}_{\epsilon^{n}}.
  }
\]

\begin{prp}\label{prp:cofibrations}
  Let $\K$ be either $\I$ or $\J$, and
  $K^{n-1}_{\epsilon}$ be $\partial I^{n}$ or $J^{n-1}_{\epsilon^{n-1}}$ according as $\K$ is $\I$ or $\J$.
  Then for any $\epsilon$-admissible map $\phi_{\epsilon} \colon K^{n-1}_{\epsilon} \to X$ the inclusion 
  $i \colon X \to X \cup_{\phi_{\epsilon}} \tilde{I}^{n}_{\epsilon^{n}}$ is a $\K$-cofibration.
\end{prp}

\begin{proof}
  Let $p \colon E \to B$ be a $\K$-fibration and let
  $f \colon X \to E$ and $g \colon Y \to B$ be smooth maps
  satisfying $p \circ f=g \circ i$.  Then we have a commutative
  diagram
  \begin{eqnarray}
    \xymatrix{%
      K^{n-1}_{\epsilon} 
      \ar[r]^{=} \ar[d] 
      &
      K^{n-1}_{\epsilon} 
      \ar[r]^-{\phi_{\epsilon}} 
      \ar[d]^-{{T^{n}_{\epsilon^{n}}}} 
      &
      X 
      \ar[d]^-{i} 
      \ar[r]^-{f} 
      &
      E 
      \ar[d]^-{p}
      \\
      I^{n} 
      \ar[r]^{{T^{n}_{\epsilon^{n}}}} 
      &
      \tilde{I}^n_{\epsilon^n} 
      \ar[r]^-{\Phi} 
      &
      X \cup_{\phi_{\epsilon}} \tilde{I}^n_{\epsilon^n}
      \ar[r]^-{g} 
      & 
      B.
    }
  \end{eqnarray}
  %
  Since $p$ is a $\K$-fibration, and since $f \circ \phi_{\epsilon}$ and
  $g \circ \Phi \circ {T^{n}_{\epsilon^{n}}}$ are $\epsilon$-admissible,
  there exists by Lemma \ref{lmm:equivalent of fibrations} {(2)} an $\epsilon$-admissible (hence $\epsilon^{n}$-tame)
  lift $h' \colon I^{n} \to E$ making the diagram commutative,
  which in turn induces by Lemma \ref{lmm:property of tame}, 
  a smooth map
  $\tilde{h}' \colon \tilde{I}^{n}_{\epsilon^n} \to E$ satisfying $h^{\prime}=\tilde{h}^{\prime} \circ {T^{n}_{\epsilon^{n}}}$.
  Now,
  we have
  \begin{enumerate}
  \item 
   $h' \circ {T^{n}_{\epsilon^{n}}}|K^{n-1}_{\epsilon}=h^{\prime}|K^{n-1}_{\epsilon}=f \circ \phi_{\epsilon}$,
   and
  \item
    $p \circ \tilde{h}' = p \circ h' \circ ({T^{n}_{\epsilon^{n}}})^{-1}
    = g \circ \Phi \circ {T^{n}_{\epsilon^{n}}} \circ
    ({T^{n}_{\epsilon^{n}}})^{-1} = g \circ \Phi$.
  \end{enumerate}
  Thus, by the property of pushouts, there is a lift
  $h \colon X \cup_{\phi_{\epsilon}} \tilde{I}^n_{\epsilon^n} \to E$ such that $h \circ i = f$ and $p \circ h = g$ hold.
\end{proof}

\begin{prp}\label{prp:deformation retract}
  Suppose $n \geq 1$ and $0 < \epsilon \leq 1/2$.  If
  { $\phi_{\epsilon} \colon J^{n-1}_{\epsilon^{n-1}} \to X$} is an
  $\epsilon$-admissible map then $X$ is a deformation retract of
  $X \cup_{\phi_{\epsilon}} \tilde{I}^n_{\epsilon^n}$.
\end{prp}

\begin{proof}
  {%
  Let $R \colon I^n \to J^{n-1}$ be an
  $\epsilon^{n-1}$-approximate retraction, say
  $R = R_{\epsilon^n,\epsilon^{n-1}}$ (cf.\
  Lemma~\ref{lmm:approximate retraction exists}),
  and define $h \colon I^{n} \times I \to I^{n}$ by the
  formula
  \[
    h(t,u)=(1-u)t + uR(t), \quad (t,u) \in
    I^n \times I.
  \]
  Then the following hold.
  \begin{enumerate}
  \item For each $u \in I$,
    $h_{u}=h|I^{n} \times \{u\} \colon I^{n} \to I^{n}$ maps
    $J^{n-1}_{\epsilon^{n-1}}$ into
    $J^{n-1}_{\epsilon^{n-1}}$, and restricts to the
    identity on its $\epsilon^{n-1}$-chamber
    $J^{n-1}(\epsilon^{n-1})$.
  \item $h_{0}=1$ and $h_{1}$ is an
    $\epsilon^{n-1}$-approximate retraction of $I^{n}$ onto
    $J^{n-1}$.
  \end{enumerate}
  Let
  $\tilde{h} = {T^{n}_{\epsilon^{n}}} \circ h \circ
  ({T^{n}_{\epsilon^{n}}} \times 1) \colon I^{n} \times I \to
  \tilde{I}^{n}_{\epsilon^{n}}$.  Then $\tilde{h}_{u}$ is
  $\epsilon^{n}$-tame for all $u \in I$, and hence there
  exists by Lemma \ref{lmm:property of tame} a homotopy
  $G \colon \tilde{I}^{n}_{\epsilon^{n}} \times I \to X
  \cup_{\phi_{\epsilon}} \tilde{I}^{n}_{\epsilon^{n}}$ such
  that the diagram below is commutative:
  \[
    \xymatrix{ I^{n} \times I \ar[r]^{{T^{n}_{\epsilon^{n}}}
        \times 1} \ar[d]^{\tilde{h}} &
      \tilde{I}^{n}_{\epsilon^{n}} \times I \ar[d]^{G}
      \\
      \tilde{I}^{n}_{\epsilon^{n}} \ar[r]^(0.4){\Phi} & X
      \cup_{\phi_{\epsilon}} \tilde{I}^{n}_{\epsilon^{n} }.
    }
  \]
  But then we have
  \[
    G_{u} \circ {T^{n}_{\epsilon^{n}}} = \Phi \circ \tilde{h}_u =
    \Phi \circ {T^{n}_{\epsilon^{n}}} \circ h_{u} \circ
    {T^{n}_{\epsilon^{n}}} = i \circ \phi_{\epsilon} \circ h_{u}
    \circ {T^{n}_{\epsilon^{n}}} = i \circ \phi_{\epsilon} \
    \mbox{on} \ J^{n-1}_{\epsilon^{n-1}}
  \]
  by Lemma \ref{lmm:uniqueness criterion for tame maps},
  because
  $i \circ \phi_{\epsilon} \colon J^{n-1}_{\epsilon^{n-1}}
  \to X \cup_{\phi_{\epsilon}} \tilde{I}^{n}_{\epsilon^{n}}$
  is $\epsilon^{n-1}$-tame and
  $h_{u} \circ {T^{n}_{\epsilon^{n}}}$ restricts to the identity
  on $J^{n-1}(\epsilon^{n-1})$.  Hence there exists a map
  \[
    H \colon (X \cup_{\phi_{\epsilon}}
    \tilde{I}^{n}_{\epsilon^{n}}) \times I \to X
    \cup_{\phi_{\epsilon}} \tilde{I}^{n}_{\epsilon^{n}}
  \]
  such that the diagram below is commutative.
  \[
    \xymatrix@C50pt{ X \times I \coprod
      \tilde{I}^{n}_{\epsilon^{n}} \times I \ar[r]^(0.55){i
        \circ pr \bigcup G} \ar[d]_{(i \times 1) \bigcup
        (\Phi \times 1)} & X \cup_{\phi_{\epsilon}}
      \tilde{I}^{n}_{\epsilon^{n}}
      \\
      (X \cup_{\phi_{\epsilon}}
      \tilde{I}^{n}_{\epsilon^{n}}) \times I \ar[ru]_{H} }
  \]
  Since the vertical map
  $(i \times 1) \bigcup (\Phi \times 1)$ is a subduction,
  $H$ gives a smooth homotopy relative to $X$ such that
  $H_{0}|\tilde{I}^{n}_{\epsilon^{n}}=\Phi \circ
  {T^{n}_{\epsilon^{n}}}$ and
  $H_{1}( \tilde{I}^{n}_{\epsilon^n}) \subset X$ hold.  Now,
  we can define a retracting homotopy of
  $X \cup_{\phi_{\epsilon}} \tilde{I}^{n}_{\epsilon^{n}}$
  onto $X$ to be the composition
  $1 \simeq H_{0} \simeq H_{1}$, where $1 \simeq H_{0}$ is
  induced by the homotopy
  $h \colon \tilde{I}^{n}_{\epsilon^{n}} \times I \to
  \tilde{I}^{n}_{\epsilon^{n}}$ given by the formula:
  $h(x,u)=\Phi((1-u)x+u{T^{n}_{\epsilon^{n}}}(x))$ for
  $(x,u) \in \tilde{I}^{n}_{\epsilon^{n}} \times I$.
  (Cf.~Proposition \ref{prp:homotopy equivalent of cubes}.)
  }%
\end{proof}
We are ready to construct a factorization $X \to G^{\infty}(\K,f) \to Y$ of a smooth map $f \colon X \to Y$.
Let $K^{n-1}$ be either $\partial I^{n}$ or $J^{n-1}$ according as $\K$ is $\I$ or $\J$.
Let $S_{n}(\K,f)$ be the set of pairs of admissible maps $\phi \colon K^{n-1} \to X$ and $\psi \colon I^{n} \to Y$ satisfying $f \circ \phi = \psi|K^{n-1}$.
  Suppose $\phi$ and $\psi$ are
  $\epsilon$-admissible.  
  {%
  Then by Lemmas~\ref{lmm:property of tame} and
  \ref{lmm:equivalent of fibrations}\,(1),
  }%
  there are a smooth map
  $\tilde{\psi} \colon \tilde{I}^{n}_{\epsilon^{n}} \to Y$
  satisfying $\tilde{\psi} \circ {T^{n}_{\epsilon^{n}}}=\psi$ and
  an $\epsilon$-admissible map
  $\phi_{\epsilon} \colon K^{n-1}_{{\epsilon}} \to X$
  satisfying $\phi_{\epsilon}|K^{n-1}=\phi$.
  {%
  Now, let
  }%
\[
  G^{1}(\K,f) = \bigcup_{n\geq 0}\bigcup_{(\phi,\psi) \in {S_{n}(\K,f)}}
  X \cup_{\phi_{\epsilon}}
  \tilde{I}^n_{\epsilon^n}.
\]
Then there are natural maps $i_{1} \colon X \to G^{1}(\K,f)$ and $p_{1} \colon G^{1}(\K,f) \to Y$ induced by the inclusions $X \to X \cup_{\phi_{\epsilon}} \tilde{I}^{n}_{\epsilon^{n}}$ and 
$f \cup \tilde{\psi} \colon X \coprod \tilde{I}^{n}_{\epsilon^{n}} \to Y$,
respectively,
such that we have
\[
f=p_{1}\circ i_{1} \colon X \xrightarrow{i_{1}}G^{1}(\K,f) \xrightarrow{p_{1}}Y.
\]
This process can be repeated to construct $G^{l}(\K,f)$ for all $l>1$.
Suppose $p_{l-1} \colon G^{l-1}(\K,f) \to Y$ exists.
Then there are a space $G^{l}(\K,f)=G^{1}(\K,p_{l-1})$ together with a factorization
\[
  p_{l-1} = p_l \circ i_l \colon G^{l-1}(\K,f) \xrightarrow{i_l}
  G^l(\K,f) \xrightarrow{p_l} Y.
\]
Consequently,
we obtain a commutative diagram
\[
  \xymatrix{%
    X \ar[r]^(0.4){i_{1}} \ar[d]^{f}
    & G^{1}(\K,f) \ar[r]^{i_{2}} \ar[d]^{p_{1}}
    & G^{2}(\K,f) \ar[r]^(0.6){i_{3}} \ar[d]^{p_{2}}
    & \cdots \ar[r]^(0.4){i_{l}}
    & G^{l}(\K,f) \ar[r] \ar[d]^{p_{l}}
    & \cdots
    \\
    Y \ar[r]^{=}
    & Y \ar[r]^{=}
    & Y \ar[r]^{=}
    & \cdots \ar[r]^{=}
    & Y \ar[r]
    & \cdots.  }%
\]
which in turn induces a factorization
\[
  f = p_{\infty} \circ i_{\infty} \colon X
  \xrightarrow{i_{\infty}} G^{\infty}(\K,f)
  \xrightarrow{p_{\infty}} Y
\]
where $G^{\infty}(\K,f)=\mbox{colim}\,G^{l}(\K,f)$.
To prove Theorem \ref{thm:factorization of a smooth map},
we need to verify the following.
\begin{enumerate}
\item
$i_{\infty} \colon X \to G^{\infty}(\K,f)$ is a $\K$-cofibration.
\item
$i_{\infty} \colon X \to G^{\infty}(\J,f)$ is a weak homotopy equivalence.
\item
$p_{\infty} \colon G^{\infty}(\K,f) \to Y$ is a $\K$-fibration.
\end{enumerate}
It follows by Proposition \ref{prp:cofibrations} that the inclusion $i_{l} \colon G^{l-1}(\K,f) \to G^{l}(\K,f)$ is a $\K$-cofibration for every $l>1$,
and hence so is the composition
\[
i_{\infty} \colon X=G^{0}(\K,f) \to \mbox{colim}\,G^{l}(\K,f)=G^{\infty}(\K,f).
\]
Thus $(1)$ holds.
Moreover,
when $\K=\J$,
each $i_{l} \colon G^{l-1}(\J,f)\to G^{l}(\J,f)$ is a deformation retract,
hence a weak homotopy equivalence,
by Proposition \ref{prp:deformation retract}.
Clearly,
this implies (2).
Finally,
to prove $(3)$ we need a further lemma.
\begin{lmm}\label{lmm:property of cubical complex}
  Let $K$ be a cubical subcomplex of $I^{m}$ and $G^{\infty}$ be the
  colimit of a sequence of inclusions of diffeological spaces
  \[
    G^0 \xrightarrow{i_1} G^1 \xrightarrow{i_2} G^2
    \xrightarrow{i_3} \cdots \xrightarrow{i_l} G^l \to
    \cdots
  \]
  Then for any smooth map $f \colon K \to G^{\infty}$, there
  exists an $N > 0$ such that the image of $f$ is contained in
  $G^N$.
\end{lmm}

\begin{proof}
By the definition,
$f$ is smooth if and only if so are its restrictions to the faces of $K$.
Hence it suffices to prove the case $K=I^{n} \ (0<n \leq m).$
  Let $\sigma = \lambda^n \colon \R^n \to I^n$.
Then the composite $f \circ \sigma$ is a plot of $G^{\infty}$.
  One easily observes, by the definition of colimits in $\Diff$,
  that there exist for any $v \in I^{n} $ an open neighborhood $V_{v}$ of $v$ and a plot $P_v \colon V_v \to G^{n(v)} \ (n(v) > 0)$ such that
  $f \circ \sigma|V_v$ coincides with the composition of $P_v$
  with the inclusion $G^{n(v)} \to G^{\infty}$.  
  Since
  $I^{n} \subset \cup_{v \in {I^{n}}}V_v$ and {$I^{n}$} is compact, there exist
  $v_1,\, \cdots,\, v_k \in I^{n}$ such that
  \[
    \textstyle f(I^{n}) {=f \circ \sigma (I^{n})} \subset \bigcup_{1 \leq j \leq k} G^{n(v_{{j}})}
  \]
  holds.  Thus we have $f(I^{n}) \subset G^N$ for
  $N = \max\{n(v_{{j}}) \mid 1 \leq j \leq k \}$.
\end{proof}
To see that $p_{\infty}$ is a $\K$-fibration,
suppose $\phi \colon K^{n-1} \to G^{\infty}(\K,f)$ and $\psi \colon I^{n} \to Y$ are admissible maps satisfying $p_{\infty} \circ \phi=\psi|K^{n-1}$,
where $K^{n-1}$ is $\partial I^{n}$ or $J^{n-1}$ according as $\K$ is $\I$ or $\J$.
Then the image of $\phi$ is contained in some $G^{l}(\K,f)$ by Lemma~\ref{lmm:property of cubical complex},
and we have a commutative diagram
\[
  \xymatrix{%
    K^{n-1} \ar[r] \ar[d] %
    & K^{n-1} \ar[r]^-{\phi} \ar[d] %
    & G^{l}(\K,f) \ar[r]^-{i_{l+1}} \ar[d]^-{p_{l}} %
    & G^{l+1}(\K,f) \ar[d]^-{p_{l+1}} \ar[r] %
    & G^{\infty}(\K,f) \ar[d]^-{p_{\infty}} %
    \\
    I^{n} \ar[r]^-{{T^{n}_{\epsilon^{n}}}} %
    & \tilde{I}^{n}_{\epsilon^n} \ar[r]^-{\tilde{\psi}}
    \ar@{.>}[rru] %
    & Y \ar[r]^{=} %
    & Y \ar[r]^-{=} %
    & Y, }
\]
where $\epsilon$ is the largest constant such that both $\phi$
and $\psi$ are $\epsilon$-admissible, and $\tilde{\psi}$
satisfies $\psi = \tilde{\psi} \circ {T^{n}_{\epsilon^{n}}}$ (cf.\
Proposition~\ref{lmm:property of tame}).
As $(\phi,\psi)$ belongs to $S(n,p_l)$, 
There exist by Lemma~\ref{lmm:equivalent of fibrations} {(1)} an extension $\phi_{\epsilon} \colon {K}^{n-1}_{\epsilon^{n-1}} \to G^{l}(\K,f)$ of $\phi$ and a smooth map 
$\Phi \colon \tilde{I}^{n}_{\epsilon^{n}} \to G^{l}(\K,f) \cup_{\phi_{\epsilon}} \tilde{I}^{n}_{\epsilon^{n}} \subset G^{l+1}(\K,f)$ making the diagram commutative.
It is now clear that the resulting composition $I^{n} \to G^{\infty}(\K,f)$ gives a desired lift for the pair $(\phi,\psi)$,
showing that $p_{\infty}$ is a $\K$-fibration.

This completes the proof of Theorem~\ref{thm:factorization of a smooth map}.
%
%
%
%
%
%
%
%
%
%
%
%
\providecommand{\bysame}{\leavevmode\hbox to3em{\hrulefill}\thinspace}
\providecommand{\MR}{\relax\ifhmode\unskip\space\fi MR }
\providecommand{\MRhref}[2]{%
  \href{http://www.ams.org/mathscinet-getitem?mr=#1}{#2}
}
\providecommand{\href}[2]{#2}

\end{document}